\documentclass[12pt]{amsart}
\usepackage[utf8]{inputenc}
\usepackage[initials]{amsrefs}
\usepackage[
        colorlinks,
]{hyperref}

\usepackage{fullpage,appendix,amssymb,amsmath}

\usepackage{cancel}
\usepackage{dsfont}

\newcommand{\Q}{\mathbb Q}
\newcommand{\R}{\mathbb R}
\newcommand{\F}{\mathbb F}

\newcommand{\ord}{{\rm ord}}
\newcommand{\C}{\mathbb C}
\newcommand{\Z}{\mathbb Z}
\newcommand{\Irr}{{\rm Irr}}
\newcommand{\Gal}{{\rm Gal}}
\newcommand{\e}{{\rm e}}
\newcommand{\eps}{\varepsilon}
\renewcommand{\d}{{\rm d}}
\newcommand{\rd}{{\rm rd}}
\newcommand{\V}{\mathcal V}
\newcommand{\M}{\mathcal M}
\newcommand{\p}{\mathfrak p}

\newtheorem{X}{X}[section]
\newtheorem{corollary}[X]{Corollary}

\newtheorem{lemma}[X]{Lemma}
\newtheorem{proposition}[X]{Proposition}
\newtheorem{theorem}[X]{Theorem}

\theoremstyle{remark}
\newtheorem*{remark}{Remark}
\newtheorem{ex}{Example}

\title{Moments in the Chebotarev density theorem:\\
non-Gaussian families}
\author{R\'egis de la Bret\`eche}
\address{Université  Paris Cité, Sorbonne Université, CNRS, 
Institut de Mathématiques de Jussieu-Paris Rive Gauche,
 F-75013 Paris,  
France}
\email{regis.delabreteche@imj-prg.fr}
\author{Daniel Fiorilli}
\address{Univ. Paris-Saclay, CNRS, Laboratoire de mathématiques d'Orsay, 91405, Orsay, France.} 
\email{daniel.fiorilli@universite-paris-saclay.fr}
\author{Florent Jouve}
\address{Univ. Bordeaux, CNRS, Bordeaux INP, IMB, UMR 5251,  F-33400, Talence, France.}
\email{florent.jouve@math.u-bordeaux.fr}
\date{\today}

\dedicatory{À la mémoire de Joël Bellaïche}

\makeatletter
\@namedef{subjclassname@2020}{\textup{2020} Mathematics Subject Classification}
\makeatother


\begin{document} 

\begin{abstract}

{In this paper we investigate higher moments attached to the Chebotarev Density Theorem. Our focus is on the impact that peculiar Galois group structures have on the limiting distribution. Precisely we consider in this paper the case of groups having a character of large degree.  Under the Generalized Riemann Hypothesis, we prove in particular that there exists families of Galois extensions of number fields having doubly transitive Frobenius group for which no Gaussian limiting distribution occurs.}
%
%
%
%
%
\end{abstract}

\maketitle

\section{Introduction}

Let $L/K$ be a Galois extension of number fields, $G=\Gal(L/K)$ and let $t\colon G \rightarrow \mathbb R$ be a class function, typically $t=\mathds{1}_C$, the indicator function of a conjugacy class $C$ of $G$. The goal of this paper is twofold: we wish to understand the moments of the variance as well as the average of the $n$-th moment (those quantities will be precisely defined in~\eqref{eq:s-mom-var} and~\eqref{eq:averaged moment} respectively) of the prime counting function
\begin{equation}\label{eq:psi}
 \psi_\eta(x;L/K,t):=\sum_{\substack{\mathfrak p\triangleleft \mathcal O_K \\ m\geq 1  }} t(\varphi_\mathfrak p^m) \frac{\log (\mathcal N\mathfrak p)}{\mathcal N \mathfrak p^{\frac m2}} \eta\big(\log (\mathcal N\mathfrak p^m/x)\big)\,, 
 \end{equation}
for any test function $\eta$ in the space $\mathcal S_\delta$ defined below and assuming that $G$ has an irreducible representation of very large degree. For more background on the relevance of this type of counting function, see~\cites{Be,FJ,BF2} as well as the references therein. The reason why we are interested in those groups is that, as we will see, they give rise to non-Gaussian moments, in contrast with what was obtained in the cyclotomic case~\cite{BF2} (more precisely, $K=\Q$, $L=\Q(\zeta_q)$ and $t=\mathds{1}_{1\bmod q} - \frac 1{\phi(q)}$). In~\eqref{eq:psi}, $\varphi_\p$ (resp. $\mathcal N\mathfrak p$) denotes the Frobenius at (resp. the norm of) the prime ideal $\p$ (see~\cite{BFJ}*{(3)} for details on the ramified case) and $\eta$ is a weight function satisfying properties which we now recall.

We let $\delta>0$ and $\mathcal S_\delta\subset \mathcal L^1(\mathbb R)$ be the set of all non-trivial differentiable even $ \eta \colon\R \rightarrow \R$ such that for all $t\in \R$,
$$\eta (t) ,  \eta'(t) \ll \e^{-(\frac 12 +\delta) |t|}, $$
and moreover for all $\xi \in \R$, we have that
\begin{equation}
 0 \leq \widehat \eta(\xi) \ll (|\xi|+1)^{-1} (\log(|\xi|+2))^{-2-\delta}.
 \end{equation}
Here\footnote{Throughout the paper we will use the notation $f(x)\ll_{A}g(x)$ and $f(x)=O_{A}(g(x))$ synonymously to mean that $|f(x)|\leq C(A)|g(x)|$ for big enough $x$ and for an implicit constant $C(A)$ depending only on the sequence of parameters $A$. When implicit constants are absolute, the index $A$ will be omitted.}, the Fourier transform is defined by
 $$  \widehat \eta(\xi) := \int_{\mathbb R} \e^{-2\pi i \xi u} \eta(u) \d u\,.$$
 Finally for any $h\in\mathcal L^1(\mathbb R)$ we define
 $$
 \alpha(h):=\int_{\R} h(t)\,{\rm d}t\,.
 $$

First we notice that, with the smooth weight $\eta\in\mathcal S_\delta$, the Chebotarev density theorem reads
$$ \psi_\eta(x;L/K,t) = x^{\frac 12} \mathcal L_\eta(\tfrac 12)({ \widehat t(1)}+o(1))\qquad  (x\to\infty)\,,$$
where, denoting by $\Irr(G)$ the set of irreducible characters of $G$, we define 
$$ \widehat t(\chi):= \langle t,\chi\rangle_G = \frac 1{|G|}\sum_{g\in G} t(g)\overline{\chi(g)}\,, \qquad 
 \mathcal L_{\eta}(u):= \int_{\mathbb R} \e^{ux} \eta(x)\d x\,\qquad (\chi\in\Irr(G)).$$
Note that $\mathcal L_{\eta}(u)=\mathcal L_{\eta}(-u)$.
Secondly, it follows from an analysis as in~\cite{FJ}*{Proposition 3.18} 
that under the Riemann Hypothesis for all Artin $L$-functions $L(s,L/K,\chi)$ (an assumption that will be denoted GRH throughout the paper), the remainder term $ \psi_\eta(x;L/K,t) -  {\widehat t(1)} x^{\frac 12} \mathcal L_\eta(\tfrac 12)$ has average value equal to $\widehat \eta(0)z(L/K,t)$, where we define
$$ z(L/K,t):= \sum_{\chi \in \Irr(G)}\widehat t(\chi)\ord_{s=\frac 12}L(s,L/K,\chi)\,.$$
Therefore we are naturally led to introduce the \emph{centered} smooth counting function:
\begin{equation}\label{eq:psi*}
\psi^*_\eta(\e^u,L/K,t):=\psi_\eta(\e^u,L/K,t)-\widehat{t}(1)\e^{\frac u2} \mathcal L_\eta(\tfrac 12)- \widehat \eta(0)z(L/K,t).
\end{equation}

For $\eta\in\mathcal S_\delta$, and $u\geq 0$, we consider the $n$-th moment
\[
\M_n(u;L/K,\eta) := \frac 1{|G|} \sum_{C\in G^\sharp} |C|\big(
\psi^*_\eta(\e^u;L/K,t_C)\big)^n\,,
\]
where $G^\sharp$ is the set of conjugacy classes of $G$ and we denote $t_C=\tfrac{|G|}{|C|}\mathds{1}_C$. We have that ${\widehat {t_C}(\chi)} =\overline{ \chi(C)}$ and thus $z(L/K,t_C) \in \mathbb R$ (by the functional equation satisfied by $L(s,L/K,\chi)$).
Equivalently, we may take the average over all $g\in G$ of the $n$-th power of the error term in the Chebotarev density theorem corresponding to the conjugacy class $C_g$ of $g$. We define $\mathcal U$ to be the set of even non-trivial integrable functions $\Phi \colon \R \rightarrow \R $ such that $\Phi,\widehat \Phi \geq 0$, and we consider for $U>0$, $\Phi\in\mathcal U$, $\eta\in\mathcal S_\delta$, $n=2$ and $s\in \mathbb N$, the moments of the variance
\begin{equation}\label{eq:s-mom-var} \V_{2,s}(U;L/K,\Phi,\eta):= \frac 1{U \int_0^{\infty} \Phi} \int_0^\infty \Phi(\tfrac uU)\Big( \M_2(u;L/K,\eta) - m_2(L/K;\eta) \Big)^s \d u\,,
\end{equation}
where
\begin{equation}\label{eq:m2}
m_2(L/K;\eta) := \lim_{U\rightarrow\infty} \frac 1U \int_2^U \M_2(u;L/K,\eta) \d u\,.
\end{equation}

We will show that under Artin's holomorphy  conjecture for $L(s,L/K,\chi)$ for all non-trivial $\chi\in\Irr(G)$ (which we will denote by AC) and GRH, the integral in~\eqref{eq:s-mom-var} converges and the limit defining $m_2$ exists. Note that in the statement of Theorem~\ref{theorem main} we assume AC for the arbitrary Galois extension $L/K$ while in Theorem~\ref{theorem main 2} AC is known to hold for the particular type of Galois extensions considered.

The goal of this paper is to study the special case when $\Irr(G)$ has one irreducible character of ``large degree'', that is there exists $\vartheta\in\Irr(G)$ such that $ \vartheta(1)^2 \geq 3|G|/4$. Note that if there exists an irreducible character $\vartheta$ of $G$ such that $2\vartheta(1)^2>|G|$ then it is unique and real. Indeed if $\vartheta_1$ and $\vartheta_2$ are distinct, irreducible, and satisfy this large degree property (\emph{e.g.} if $\vartheta_1$ is a non real character of this type and $\vartheta_2=\overline{\vartheta_1}$) then we would get $|G|<\vartheta_1(1)^2+\vartheta_2(1)^2\leq\sum_{\chi\in\Irr(G)}\chi(1)^2=|G|$; a contradiction.
We introduce the following notation: 
\begin{equation}\label{eq:lambdaj}
\lambda_j(G):=\sum_{\chi \in \Irr(G)} \chi(1)^j\,,\qquad (j\geq 1)\,.
\end{equation}


Our first main result establishes a lower bound on the $s$-th moment of the variance~\eqref{eq:s-mom-var}. We let
$$\mu_0:=1\,,\qquad \mu_{r} := \begin{cases} (2n-1) \cdot (2n-3) \cdots 1 & \text{ if } r=2n\geq 2,\\
0 & \text{ if } r=2n+1.
\end{cases}$$ 
denote $r$-th moment of the Gaussian, and for a number field $L$, we let ${\rm rd}_L:=d_L^{\frac 1{[L:\Q]}}$, where $d_L$ is the absolute value of the absolute discriminant of $L$. We also recall the Odlyzko bound stating that there exists an absolute constant $C>0$ such that ${\rm rd}_L \geq C$ (\cite{Odl}*{p. 120}). As all results in the literature on moments of prime-counting functions, this suggests a Gaussian distribution for prime ideals in Chebotarev classes. However, note that in our second main result (see Theorem~\ref{theorem main 2} below), we will construct a framework allowing to find explicit families in which Gaussian distribution does not occur.

\begin{theorem}
\label{theorem main}
Let $L/K$ be a Galois extension and assume AC and GRH. Let $\eta\in \mathcal S_\delta$ and $\Phi\in \mathcal U$. Assume moreover that there exists a (necessarily real) character $\vartheta\in\Irr(G)$ such that $ |G|-f(|G|) \leq \vartheta(1)^2 \leq |G|-1 $ where $f$ is a function such that $1\leq f(|G|)\leq \frac14{|G|}$. Then for $s\in \mathbb N$ and $U \geq 1$ we have the estimate
\begin{multline*}
    \V_{2,s}(U;L/K,\Phi,\eta)\geq H_s \big(\alpha(|\widehat \eta|^2)[K:\Q] |G|^{\frac 12}\log(\rd_L) \big)^s\Big( 1+O\Big(\frac{sf(G)^{\frac 12}}{|G|^{\frac 12}}+ \frac{ \kappa_\eta^s }{ \log\log(\rd_L+2)}\Big)\Big) 
    \\
+O_{\Phi}\Big( \frac{(\kappa_\eta[K:\Q]^2 (\log (\rd_L+2))^2 
\lambda_1(G)^2)^{s} }U\Big),
\end{multline*} 
where
 \begin{equation}
      H_s:=\sum_{j=0}^s \binom{s}{j}(-1)^{s-j} \mu_{2j},
      \label{def Hs}
 \end{equation}
 and where $\kappa_\eta>0$ is a constant which depends on $\eta$.
\end{theorem}

\begin{remark} The quantity
$H_s$ is the $s$-th moment of $Z^2-1$, where $Z\sim {\mathcal N}(0,1)$, the standard Gaussian. Indeed, the $s$-th moment of $Z^2-1$ is equal to
$$ \mathbb E \big[(Z^2-1)^s\big] = \frac 1{\sqrt{2\pi}} \int_{\mathbb R} (t^2-1)^s \e^{-\frac{t^2}{2}} \d t = \sum_{j=0}^s \binom{s}{j}(-1)^{s-j} \frac 1{\sqrt{2\pi}} \int_{\mathbb R} t^{2j} \e^{-\frac{t^2}{2}} \d t =H_s.  $$
Note that $H_1=0$ and morever $H_{2k+1} >0$ for all $k\geq 1$ (this follows from the recurrence $H_{s+2} = 2(s+1)(H_{s+1}+H_{s}) $ for $s\in \mathbb N$).
\end{remark}

We now give an intuitive reason why we should expect the moments $H_s$ in Theorem~\ref{theorem main}.
By Parseval's identity, we have that  
\begin{equation}
 \M_2(u;L/K,\eta) = \sum_{\chi \in \Irr(G)} \big|\psi_\eta^*(\e^u;L/K,\chi)\big|^2. 
 \label{eq:parseval}
\end{equation}
In the setting of Theorem~\ref{theorem main}, we expect this sum to be dominated by the contribution of the (unique) character $\chi=\vartheta$ of large degree. Moreover, we expect $\psi_\eta^*(\e^u;L/K,\chi)$ to have Gaussian moments (see~\cite{FJ},~\cite{BFJ}), hence we expect $\M_2(u;L/K,\eta)$ to be approximately the square of a Gaussian. As a result, we expect $\V_{2,s}(U;L/K,\Phi,\eta)$ to be the centered moments of the square of a Gaussian, which is what we obtain.

In addition to computing the higher centered moments of the variance $\M_2(u;L/K,\eta) $, we also compute the first moment of all moments $\M_n(u;L/K,\eta)$ for particular non-abelian Galois groups $G$.
More precisely we will be interested in the moment
\begin{equation}
     M_{n,1} (U;L/K,\Phi,\eta):= \frac 1{U \int_0^\infty \Phi}  \int_0^\infty \Phi(\tfrac uU) \M_n (u ;L/K,\eta)
\d u. 
\label{eq:averaged moment}
\end{equation}

The cyclotomic case where $G\simeq (\Z/q\Z)^\times$ was studied in~\cite{BF1}; in this context  we expect to have Gaussian moments.  Our next result focuses on the case where the group $G$ has order $d(d+1)$ for some integer $d\geq 2$, and where $G$ admits an irreducible representation of degree $d$ (specifically $G$ is a doubly transitive Frobenius group as defined \textit{e.g.} in~\cite{Hup}*{Prop. 6.5},
see also~\S\ref{section:LargeDegree}). 
We will show that for this family of groups, the associated moments are not necessarily Gaussian. 

\begin{theorem}
\label{theorem main 2}
Let $G$ be a doubly transitive Frobenius group of order $d(d+1)\geq 6$ with kernel $N$ and abelian complement $H$. Let $L/K$ be a Galois extension of number fields of group $G$ and assume GRH. Let $\eta\in \mathcal S_\delta$ an $\Phi\in \mathcal U$.  Then for $m\in \mathbb N$ and $U\geq 1$ and assuming that $\rd_L \geq \e^{Cm}$ with $C>0$ a large enough constant, we have the estimate
\begin{multline}\label{eq:borneinfTh1.2}
    M_{2m,1}(U;L/K,\Phi,\eta)\geq  \frac{\mu_{2m}d^{2m}}{d(d+1)}\widehat \eta^{2m}(0) \Big((d-1) [K:\Q]\log (\rd_L)\Big)^m 
    \Big(1+O_{\eta}\Big(\frac m{\log(\rd_L+2) }  \Big)\Big)\hspace{-1cm}\\ +O_{\Phi}\Big(\frac{ (d+1)^{4m-2} (\kappa_\eta [K:\Q] \log(\rd_L+2))^{2m}}U\Big),
\end{multline} 
where $\kappa_\eta>0$ is a large enough constant depending on $\eta$.

Moreover, the normalized moment satisfies the following asymptotic lower bound\footnote{Note that we expect the ``true size'' of this moment to be asymptotically $\mu_{2m} d^{2m} $.}
$$ \liminf_{U\rightarrow \infty}  \frac{M_{2m,1}(U;L/K,\Phi,\eta)} { M_{2,1}(U;L/K,\Phi,\eta)^m} \geq (1+o_{d\rightarrow \infty}(1))\frac{\mu_{2m}d^{m-2}}{( [K:\Q]\log(\rd_L) )^m }, $$
and therefore these moments are non-Gaussian as soon as $[K:\Q]\log(\rd_L) = o(d^{1-\eps})$ for some fixed $\eps>0$. (In both $o()$ terms above, $m$ is fixed.)

\end{theorem}

To obtain an explicit non-Gaussian family,  consider the extension (see Example~\ref{ex:rad}) $K=\Q$, $L=\Q(a^{\frac 1p},\zeta_p)$ where $a,p$ are primes such that $p^2\nmid a^{p-1}-1$  (there is a sequence of such pairs $(a,p)$ with both $a$, $p$ tending to infinity; see~\cite{FJ}*{Lemma 9.9}), one has that $\log(\rd_L) \ll \log(ap)$ and $d=p-1$ (see~\cite{FJ}*{Prop. 9.8} where the table of characters of $\Gal(L/K)\simeq (\Z/p\Z)\rtimes(\Z/p\Z)^\times$ as well as the formula for the discriminant of $L$ are stated). We conclude that $[K:\Q]\log(\rd_L) =\log(\rd_L)= o(d^{1-\eps})$, and thus Theorem~\ref{theorem main 2} implies that we have non-Gaussian moments.


\begin{remark}
All the irreducible representations of a doubly transitive Frobenius group $G$ are monomial (\textit{i.e.} induced from a representation of degree $1$ of a subgroup, see~Corollary~\ref{cor:doubtrans}). In particular Brauer's Theorem implies that AC holds for any Galois extension of number fields whose Galois group is isomorphic to $G$.
\end{remark}

The paper is organized as follows. In~\S\ref{section:prep} we prove Theorem~\ref{theorem main} by first establishing a few preparatory results  about properties of $H_s$ as well as bounds on certain Artin conductors. We devote~\S\ref{section:LargeDegree} to the proof of Theorem~\ref{theorem main 2} which exploits specific properties of doubly transitive Frobenius groups.

\section{Moments of the variance: Proof of Theorem~\ref{theorem main}}
\label{section:prep}
Let us first give a precise bound on the Artin conductor $A({\vartheta})$ (see \textit{e.g.}~\cite{BFJ}*{\S3} for recollections on Artin conductors in this context) of the irreducible character of large degree~$\vartheta$ appearing in Theorem~\ref{theorem main}.

\begin{lemma}
\label{lemma evaluation Artin conductor}
Let $L/K$ be a Galois extension of number fields of group $G$ where $|G| \geq 4$, and assume that
there exists $\vartheta\in\Irr(G)$ such that $ |G|- f(|G|) \leq \vartheta(1)^2 \leq |G|-1 $, where $f$ is a function such that $ 1\leq f(|G|)\leq \frac14 {|G|} $. Then we have the bounds
$$ 1- \frac{f(|G|)^{\frac 12}}{|G|^{\frac 12}} -2\frac{f(|G|)^{\frac 32}}{|G|^{\frac 32}}    \leq  \frac{\log(A({\vartheta}))}{\vartheta(1) [K:\Q]\log(\rd_L)}   \leq  1+ \frac{f(|G|)^{\frac 12}}{|G|^{\frac 12}} +2\frac{f(|G|)^{\frac 32}}{|G|^{\frac 32}}.
$$
 
\end{lemma}

\begin{proof}
Since $$ 1= \langle \vartheta,\vartheta\rangle_G = \frac{\vartheta(1)^2}{|G|}+ \frac 1{|G|} \sum_{1\neq g\in G} \vartheta(g) \overline{\vartheta(g)}, $$
it follows that for all $g\neq 1$, 
$$  |\vartheta(g)| \leq f(|G|)^{\frac 12} \leq     \vartheta(1)\frac{f(|G|)^{\frac 12}}{|G|^{\frac 12}} \Big( 1+2\frac{f(|G|)}{|G|}\Big).$$
As a result,~\cite{FJ}*{Lemma 4.2} implies the claimed bounds.
\end{proof}

\begin{remark}
Writing $d=\vartheta(1)$, we have $|G|=d(d+e)$ for some $e\geq 1$ (see the beginning of \S\ref{section:LargeDegree}). As a result, in Lemma~\ref{lemma evaluation Artin conductor} one necessarily has $f(|G|) \geq de\geq |G|^{\frac 12}-\frac 12$. 
\end{remark}
 We will need the following combinatorial result.
\begin{lemma}
\label{lemma combinatorics}
Fix $s\in \mathbb N$ and let ${\mathcal H}_s$ be the set of involutions $ \pi: \{1,\dots,s\}\times \{1,2\} \rightarrow \{1,\dots,s\}\times \{1,2\} $ having no fixed point and such that $\pi(j,1) \neq (j,2)$ for all $j\leq s$. Then we have 
$
  | \mathcal H_s|=H_s
  $,
where $H_s$ is defined by~\eqref{def Hs}.
\end{lemma}

\begin{remark} As $s\rightarrow \infty,$ we have the asymptotic formula
$$H_{s}=(\e^{-1/2}+o(1))\mu_{2s}.$$
 \end{remark}

\begin{proof}[Proof of Lemma~\ref{lemma combinatorics}]
We first prove the relation
\begin{equation}\label{sumF=mu}
\sum_{k=0}^s \binom{s}{k}|\mathcal H_{s-k}|=\mu_{2s},
\end{equation}
from which the claimed identity follows. Consider the set of involutions $ \pi: \{1,\dots,s\}\times \{1,2\} \rightarrow \{1,\dots,s\}\times \{1,2\} $ having no fixed points; the total number of such involutions is $\mu_{2s}$. This set can be partitioned according to the number $k$ of $1\leq j\leq s$  such that $\pi(j,1) =  (j,2)$. The number of elements in each of these subsets is equal to $\binom{s}{k}|\mathcal H_{s-k}| $, thus we obtain the claimed identity. The proof is finished.
\end{proof}
 
The next step will be to apply the explicit formula and bound $\V_{2,s}(U;L/K,\Phi,\eta)$ in terms of the quantity
\begin{equation}
\label{def mathbb V 2s}
\mathbb V_{2,s}(L/K;\eta ):=
\sum_{\chi_1,\dots,\chi_s \in \Irr(G)}  \sum_{ \substack{\gamma_{\chi_1} ,\gamma_{\chi_1}' ,\dots ,\gamma_{\chi_s} ,\gamma_{\chi_s}' \neq 0\\ \gamma_{\chi_j}\neq \gamma_{\chi_j}' \\ \sum_{j=1}^s(\gamma_{\chi_j}-\gamma_{\chi_j}')=0 }} \prod_{j=1}^s\Big( \widehat \eta\Big( \frac {\gamma_{\chi_j}}{2\pi } \Big)\widehat \eta\Big( \frac {\gamma_{\chi_j}'}{2\pi } \Big) \Big),
\end{equation}
where $\gamma_{\chi_j}$ and $\gamma_{\chi_j}'$ run through the imaginary parts of the non-trivial zeros of $L(s,L/K,\chi_j)$.

\begin{lemma}
\label{lemma first bound V}
Fix $s\in \mathbb N$ and let $L/K$ be a Galois extension and assume AC and GRH. Let $\eta\in \mathcal S_\delta,\Phi\in \mathcal U$. 
Then for $U\geq 1$ we have the estimate 
$$
\V_{2,s}(U;L/K,\Phi,\eta)\geq 
\mathbb V_{2,s}(L/K;\eta ) 
+O_{ \Phi}\Big( \frac{(\kappa_\eta[K:\Q]^2 (\log (\rd_L+2))^2 
\lambda_1(G)^2)^{s} }U\Big),$$
where $\kappa_\eta>0$ is a constant which depends on $\eta$.
\end{lemma}


\begin{proof}

Recall the notation $t_C=\tfrac{|G|}{|C|}\mathds{1}_C$.
By~\cite{BFJ}*{Lemma 4.1}, we have that
\[ \psi_\eta^*(\e^u;L/K,t_C)=-\sum_{\chi \in \Irr(G)} \overline{\chi}(C)\sum_{\gamma_\chi\neq  0} \e^{i \gamma_\chi u} \widehat \eta\Big( \frac {\gamma_\chi}{2\pi } \Big)+O_{\eta}\Big( \e^{-\frac u2}  [K:\Q] \lambda_{1,1}(t_C) \log(\rd_L+2) \Big),
\]
where $\psi_\eta^*$ is defined by~\eqref{eq:psi*}, $\gamma_\chi$ runs over the imaginary parts of the non real zeros of $L(s,L/K,\chi)$ and where $\lambda_{1,1}(t_.)\colon G^\sharp\to\R$ is defined by
$
\lambda_{1,1}(t_C)=\sum_{\chi \in \Irr(G)}\chi(1)|\widehat{t_C}(\chi)|$.

Note that the main term here satisfies the bound
$$\sum_{\chi \in \Irr(G)} \overline{\chi}(C)\sum_{\gamma_\chi\neq  0} \e^{i \gamma_\chi u} \widehat \eta\Big( \frac {\gamma_\chi}{2\pi } \Big) \ll_{\eta}  [K:\Q] \lambda_{1,1}(t_C) \log(\rd_L+2),$$
by the Riemann--von Mangoldt formula~\cite{IK}*{Theorem 5.8}. By the orthogonality of irreducible characters, we deduce that
\begin{multline*}
\M_2(u;L/K,\eta) =  \sum_{\chi \in \Irr(G)} \sum_{ \substack{\gamma_{\chi} , \gamma_{\chi}' { \neq 0}}}\e^{i (\gamma_{\chi}-\gamma_{\chi}') u} \widehat \eta\Big( \frac {\gamma_{\chi}}{2\pi } \Big) \widehat \eta\Big( \frac {\gamma_{\chi}'}{2\pi } \Big) \\+O_{ \eta}\Big(\e^{-\frac u2}\frac{[K:\Q]^2  (\log(\rd_L+2))^2 }{ |G|}\sum_{C\in G^\sharp} |C|\lambda_{1,1}(t_C)^2 \Big),
\end{multline*} 
which implies in particular that
(recall~\eqref{eq:m2})
$$ m_2(L/K;\eta) =  \sum_{\chi \in \Irr(G)} b_0^+(\chi;|\widehat\eta|^2),$$
 with
 \begin{equation}
  \label{def b+}
    b_0^+(\chi;h):=\sum_{\rho_\chi\notin \mathbb R} {\rm ord}_{s= \rho_\chi} L(s,L/K,\chi) h\Big(\frac{\rho_\chi-\frac 12}{2\pi i }\Big).
  \end{equation}
 

Note also that (recalling~\eqref{eq:lambdaj})
\[ 
\sum_{C\in G^\sharp} |C|\lambda_{1,1}(t_C)^2=\frac 1{|G|}\sum_{C\in G^\sharp} |C| \Big(  \sum_{\chi \in \Irr(G)} \chi(1)|\chi(C)|\Big)^2 \leq \Big( \sum_{\chi \in \Irr(G)} \chi(1)\Big)^2=\lambda_1(G)^2\,.
\]
We then have the following bound for the main term in our estimate for $\M_2(u;L/K,\eta) $:
\begin{align*}
 \sum_{\chi \in \Irr(G)} \sum_{ \substack{\gamma_{\chi} , \gamma_{\chi}'\neq 0   }}\e^{i (\gamma_{\chi}-\gamma_{\chi}') u}\widehat \eta\Big( \frac {\gamma_{\chi}}{2\pi } \Big)\widehat \eta\Big( \frac {\gamma_{\chi}'}{2\pi } \Big)
 &\ll_{\eta} \sum_{\chi \in \Irr(G)} (\log A(\chi))^2 
 \ll [K:\Q]^2 (\log (\rd_L+2))^2 |G|,
\end{align*} where we have used~\cite{BFJ}*{Lemma 4.2}  as well as~\cite{FJ}*{Lemma 4.1}. 
Expanding the $s$-th power, we deduce that
 \begin{align*}   
 \V_{2,s}&(U;L/K,\Phi,\eta) \cr
 &=  \frac 2{U\widehat \Phi(0)}\int_0^\infty \Phi(\tfrac uU)\Big(    \sum_{\chi \in \Irr(G)} \sum_{ \substack{\gamma_{\chi} , \gamma_{\chi}' \neq 0 \\ \gamma_{\chi} \neq  \gamma_{\chi}'  }}\e^{i (\gamma_{\chi}-\gamma_{\chi}') u}\widehat \eta\Big( \frac {\gamma_{\chi}}{2\pi } \Big)\widehat \eta\Big( \frac {\gamma_{\chi}'}{2\pi } \Big) \\
 &\hspace{2cm} +O_{ \eta}\Big(\e^{-\frac u2}  \big( [K:\Q]   \log (\rd_L+2) 
\lambda_1(G)^2\big)^{2}   \Big)\Big)^s \d u, \cr
 &
=\frac { 1}{\widehat \Phi(0)}\sum_{\chi_1,\dots,\chi_s \in \Irr(G)}  \sum_{ \substack{\gamma_{\chi_1} ,\gamma_{\chi_1}' ,\dots ,\gamma_{\chi_s} ,\gamma_{\chi_s}' { \neq 0 } \\ \gamma_{\chi_j}\neq \gamma_{\chi_j}'  }} \widehat \Phi\Big( \frac{U \sum_{j=1}^s(\gamma_{\chi_j}-\gamma_{\chi_j}')}{2\pi} \Big) \prod_{j=1}^s\Big( \widehat \eta\Big( \frac {\gamma_{\chi_j}}{2\pi } \Big)\widehat \eta\Big( \frac {\gamma_{\chi_j}'}{2\pi } \Big) \Big) 
\cr &\hspace{2cm}
+O_{\Phi}\Big( \frac{(\kappa_\eta[K:\Q]   \log (\rd_L+2)  
\lambda_1(G)^2)^{2s}  }U\Big).
 \end{align*}

The claimed bound follows since by positivity of $\widehat \Phi$ and $\widehat \eta$ we may remove any number of terms in the summation.
\end{proof}


We now give a lower bound for  $\mathbb V_{2,s}(L/K;\eta )$  by noting that we may discard any number of terms in its defining sum, since each summand is non-negative.

\begin{lemma}\label{lemme V2s ineg}
Let $L/K$ be a Galois extension and assume AC and GRH. Let $\eta\in \mathcal S_\delta,\Phi\in \mathcal U$. Assume moreover that there exists a (necessarily real) character $\vartheta\in\Irr(G)$ such that $ |G|-f(|G|) \leq \vartheta(1)^2 \leq |G|-1 $ where $f$ is a function such that $1\leq f(|G|)\leq \frac14{|G|}$. Then for $s\in \mathbb N$ we have the bound
 $$\mathbb V_{2,s}(L/K;\eta) \geq H_s \big(\alpha(|\widehat \eta|^2)[K:\Q] |G|^{\frac 12}\log(\rd_L) \big)^s\Big( 1+O\Big(\frac{sf(G)^{\frac 12}}{|G|^{\frac 12}} + \frac{\kappa_\eta^s}{ \log\log(\rd_L+2)}\Big)\Big), $$
 where $H_s$ is defined by~\eqref{def Hs}.
\end{lemma}

\begin{proof}[Proof of Lemma~\ref{lemme V2s ineg}]
Firstly, for $s=1$ the claimed bound is trivial since $\mathbb V_{2,s}(L/K;\eta) =0$.
For $s\geq 2$, restricting the sum in~\eqref{def mathbb V 2s} to the character $\vartheta$, we have the lower bound 
\begin{equation}
    \mathbb V_{2,s}(L/K;\eta) \geq V_s(\vartheta;\eta):= \sum_{ \substack{\gamma_{1,1} ,\gamma_{1,2} ,\dots ,\gamma_{s,1} ,\gamma_{s,2} \neq 0\\ \gamma_{j,1}\neq \gamma_{j,2} \\ \sum_{j=1}^s(\gamma_{j,1}-\gamma_{j,2})=0 \\\forall i\leq s  \forall k \in \{1,2\}\exists! j \neq i: \gamma_{i,k} \in \{ -\gamma_{j,1}, \gamma_{j,2}\} }} \prod_{j=1}^s\Big( \widehat \eta\Big( \frac {\gamma_{j,1}}{2\pi } \Big)\widehat \eta\Big( \frac {\gamma_{j,2}}{2\pi } \Big) \Big), 
    \label{equation lemma wedding}
\end{equation} 
where the $\gamma_{j,k}$ run through the imaginary parts of the non-trivial zeros of $L(s,L/K,\vartheta)$. 
We define $\mathcal H_s$ to be the set of involutions $ \pi: \{1,\dots,s\}\times \{1,2\} \rightarrow \{1,\dots,s\}\times \{1,2\} $ having no fixed point and such that $\pi(j,1) \neq (j,2)$ for all $j\leq s$. We have $H_s=|\mathcal H_s|$. Given $\pi \in \mathcal H_s,$ we define
$$ \Gamma_\pi :=\left\{(\gamma_{1,1},\gamma_{1,2},\dots , \gamma_{s,1},\gamma_{s,2}) \colon\!\!\!
\begin{array}{ll} & \forall (j,k), \gamma_{j,k}\neq 0, \hfill
\\ &\forall (j,k),(i,\ell),
(\gamma_{j,k}+(-1)^{\ell-k}\gamma_{i,\ell}=0)\Rightarrow  (i,\ell)=\pi(j,k)
\end{array}\right\} .   $$

Now, given $\pi\in \mathcal H_s$, 
we note that the vector $(\gamma_{1,1},\gamma_{1,2},\dots , \gamma_{s,1},\gamma_{s,2}) \in \Gamma_\pi$  where for all $(j,k)$, $\gamma_{j,k}=\eps_{j,k}\gamma_{\pi(j,k)}$ and where by definition $\eps_{j,k}=-1$ if $\pi(j,k)_2 =k $ and $\eps_{j,k}=1$ otherwise, corresponds to a summand on the right hand side of~\eqref{equation lemma wedding}. Note that if $\pi_1,\pi_2$ are distinct involutions, then the corresponding vectors of zeros are distinct. We deduce that 
$$
V_s(\vartheta;\eta )\geq \sum_{\pi \in H_s} \sum_{ \substack{(\gamma_{1,1} ,\gamma_{1,2} ,\dots ,\gamma_{s,1} ,\gamma_{s,2}) \in \Gamma_\pi  \\  \forall j,k, \gamma_{j,k}=\eps_{j,k}\gamma_{\pi(j,k)} }} \prod_{j=1}^s\Big( \widehat \eta\Big( \frac {\gamma_{j,1}}{2\pi } \Big)\widehat \eta\Big( \frac {\gamma_{j,2}}{2\pi } \Big) \Big). $$
Reindexing, we see that the innermost sum is 
\begin{align*}
&\geq \sum_{\gamma_1} \sum_{\gamma_2\notin\{\gamma_1,-\gamma_1\} } \ldots
\sum_{\gamma_s\notin\{\gamma_1,-\gamma_1,\ldots ,\gamma_{s-1},-\gamma_{s-1} \} }
\prod_{j=1}^s\Big(m({\tfrac 12+i}\gamma_j)\Big| \widehat \eta\Big( \frac {\gamma_{j}}{2\pi } \Big)\Big|^2\Big)\\ 
&\geq \sum_{\gamma_1} \sum_{\gamma_2\notin\{\gamma_1,-\gamma_1\} } \ldots
\sum_{\gamma_s\notin\{\gamma_1,-\gamma_1,\ldots ,\gamma_{s-1},-\gamma_{s-1} \} }
\prod_{j=1}^s \Big| \widehat \eta\Big( \frac {\gamma_{j}}{2\pi } \Big)\Big|^2 \\ 
&\geq
b_0(\vartheta; |\widehat\eta|^2)^{s-2}\big(b_0(\vartheta; |\widehat\eta|^2)^{2}+O_{ \eta}(s^2 b_0^+(\vartheta; |\widehat\eta|^4))\big),
\end{align*}
where the $\gamma_j$ run through the imaginary parts of the non-trivial zeros of $L(s,L/K,\vartheta)$ and $m(\tfrac 12+i\gamma_j)$ denotes the multiplicity of the corresponding zero, and where 
(recall~\eqref{def b+})
\begin{equation*}
  b_0(\chi;|\widehat \eta|^2):=\sum_{\rho_\chi\notin \mathbb R} \Big| \widehat \eta\Big(\frac{\rho_\chi-\frac 12}{2\pi i }\Big)\Big|^2,\qquad b_0^+(\chi;|\widehat \eta|^4)=\sum_{\rho_\chi\notin \mathbb R} m(\rho_\chi) \Big| \widehat \eta\Big(\frac{\rho_\chi-\frac 12}{2\pi i }\Big)\Big|^4\,. \end{equation*}
 By the bound~\cite{CCM}*{Theorem 5} on multiplicities of zeros, the total contribution of the error term is 
$$ \ll_{\eta} H_s s^2\frac{ ( {\kappa_\eta} \log(A(\vartheta)))^{s}}{\log\log A(\vartheta)^{\frac 3{\vartheta(1)}}} \ll H_s s^2  \frac{(  {\kappa_\eta}\vartheta(1) [K:\Q] \log(\rd_L+2))^s}{\log\log(\rd_L+2)}, $$
by Lemma~\ref{lemma evaluation Artin conductor}. Denoting
\[
b(\chi;h):=\sum_{\rho_\chi}  h\Big(\frac{\rho_\chi-\frac 12}{2\pi i }\Big)\,,
\]
we use the same bound~\cite{CCM}*{Theorem 5} on multiplicities in order to replace $b_0(\vartheta;|\widehat \eta|^2) $ with $b(\vartheta;|\widehat \eta|^2)$. Indeed the difference between these two quantities is
$\ll H_s s^2 \frac{( {\kappa_\eta}\vartheta(1) [K:\Q] \log(\rd_L+2))^s}{\log\log(\rd_L+2)}.$
The proof is finished by  applying~\cite{BFJ}*{Lemma 4.2} as well as Lemma~\ref{lemma combinatorics}.
\end{proof}
We are now ready to prove Theorem~\ref{theorem main}.

\begin{proof}[Proof of Theorem~\ref{theorem main}]
The proof is obtained by combining Lemmas~\ref{lemma first bound V} and ~\ref{lemme V2s ineg}.
\end{proof}

\section{Groups having an irreducible character of maximal degree: Proof of Theorem~\ref{theorem main 2}} \label{section:LargeDegree}

In this section we focus on a specific case of groups $G=\Gal(L/K)$ having an irreducible character of maximal degree (with respect to the order of the group). Our goal will be to establish a lower bound on the averaged moment $ M_{n,1} (U;L/K,\Phi,\eta)$ (recall~\eqref{eq:averaged moment}) 
and prove Theorem~\ref{theorem main 2}, which will allow us to deduce non-Gaussian behaviour in families of such Galois extensions.

Let $G$ be a non-trivial finite group and let $\chi$ be an irreducible representation of $G$ of degree $d=\chi(1)$. It is a classical fact form the representation theory of finite groups over $\C$ that $d\mid |G|$ and $d^2\leq |G|$. Therefore one may write $|G|=d(d+e)$ for some positive integer $e$. Note that if $e<d$ then $G$ has a unique irreducible character of degree $d$. 



In~\cite{Sny}, Snyder classifies the groups for which $e=1$ and shows that for fixed $e\geq 2$ there are only finitely many (with an explicit bound depending only on $e$) isomorphism classes of groups $G$ such that $|G|=d(d+e)$. In order to state Snyder's result, we recall the definition of \textit{Frobenius group} (see~\cite{Hup}*{Prop. 16.5}). Let $G$ be a finite group having a normal subgroup $N$ and a subgroup $H\not\in\{\{1\},G\}$ such that $G=NH$. The group $G$ is a \textit{Frobenius group of kernel $N$ and complement $H$} if, for all $h\in H\smallsetminus\{1\}$, conjugation by $h$ acts without fixed points on $N\smallsetminus\{1\}$. If in addition the action of $H$ on $N\smallsetminus\{1\}$ is transitive, one says that $G$ is a \textit{doubly transitive} Frobenius group.

For $e=1$, Snyder's result is the following (see~\cite{Sny}*{Prop. 2.1}).



\begin{proposition}[Snyder] \label{prop:snyder}
Let $G$ be a finite group. The following assertions are equivalent.
\begin{enumerate}
\item There exists $\chi\in\Irr(G)$ of degree $d$ such that $d(d+1)=|G|$,
\item Either $G\simeq \Z/2\Z$ or $G$ is a doubly transitive Frobenius group \textit{i.e.} it is isomorphic to the semi direct product of a subgroup $N$ of order $d+1$ by a subgroup $H$ of order $d$ that acts freely and transitively by conjugation on $N\smallsetminus\{1\}$. 
\end{enumerate}
Moreover if $|G|>2$ the character $\chi$ in (1), which is necessarily unique, is induced from any non trivial character of degree $1$ of the subgroup $N$ of (2).

\end{proposition}

Note that the characterization of doubly transitive Frobenius groups mentioned in Proposition~\ref{prop:snyder}(2) is straightforward (indeed for such a group
$N\cap H$ is trivial, from the fact that conjugation by any $h\in H$ is fixed point free on $N\smallsetminus\{1\}$). 

\begin{remark}
The classification results of Snyder are more general than what is stated in Proposition~\ref{prop:snyder}. Indeed under the assumption $|G|=d(d+e)$ where $d=\lfloor\sqrt{|G|}\rfloor>1$ is the degree of some $\chi\in\Irr(G)$, we have $e\in\{1,2\}$ and Snyder proves~(\cite{Sny}*{Th. 1.1, Th. 6.2}) that $G$ is either a doubly transitive Frobenius group (case $e=1$), or a non abelian group of order $8$ (case $e=2$).

\end{remark}

In the sequel we assume $d>1$, therefore excluding the case $G\simeq \Z/2\Z$ of Proposition~\ref{prop:snyder}(2). 
We next recall~\cite{Hup}*{Th. 18.7} which describes the representation theory of Frobenius groups.

\begin{proposition} \label{prop:carGpFrob}
Let $G$ be a Frobenius group of kernel $N$ and complement $H$. The group $H$ acts on $\Irr(N)$ as follows
$$
\varphi^h:=h\cdot \varphi\colon \big(x\mapsto \varphi(hxh^{-1})\big)\qquad (\varphi\in\Irr(N), h\in H)\,.
$$
The orbits in this action are $\{\varphi_1= \mathds{1}_N\}$ as well as $\{\varphi^h\colon h\in H\}$ that has cardinality $|H|$ for any $\varphi\in\Irr(N)\smallsetminus\{ \mathds{1}_N\}$. Let $\{\varphi_j\colon 2\leq j\leq s\}$ be a system of representatives of the orbits in this action. A set of representatives for the isomorphism classes of the irreducible representations of $G$ is given by
\begin{enumerate}
\item the irreducible representations of $H$ \textit{via} $G\to G/N\simeq H$,
\item the induced representations ${\rm Ind}_N^G\varphi_j$, each of degree $|H|\varphi_j(1)$, for $2\leq j\leq s$. These representations satisfy $N=\{x\in G\colon { {\rm Tr}}\big({\rm Ind}_N^G\varphi_j(x)\big)\neq 0\}$.
\end{enumerate}
In particular $G$ has $|\Irr(H)|+(|\Irr(N)|-1)/|H|$ irreducible characters.
\end{proposition}

\begin{corollary}\label{cor:doubtrans} 
{  Let $G$ be a doubly transitive Frobenius group of kernel $N$ and complement $H$.} Assume that $H$ is abelian and non trivial. Then $G$ has only one non abelian irreducible character: it is induced by any non trivial character of degree $1$ of $N$. Moreover $N$ is then necessarily abelian and $N\smallsetminus\{1\}$ is a conjugacy class of $G$.
\end{corollary}

\begin{proof}
Let $d:=|H|>1$, hence $|G|=d(d+1)$ since $N\cap H$ is trivial and $N\smallsetminus\{1\}$ is a single orbit, with trivial stabilizer, under the action of $H$. We use the fact that $H$ is abelian combined with Proposition~\ref{prop:carGpFrob} to deduce that
$$
d(d+1)=|G|=\sum_{\chi\in\Irr(G)}\chi(1)^2=|H|+\sum_{j=2}^s|H|^2\varphi_j(1)^2=d\Big(1+d\sum_{j=2}^s\varphi_j(1)^2\Big)\,.
$$
Therefore $s=2$ and $\varphi_2(1)=1$. By Proposition~\ref{prop:carGpFrob} this implies that $G$ has a unique non abelian irreducible character defined by ${\rm Tr}\big({\rm Ind}_N^G\varphi_2\big)$. Moreover $\Irr(N)\smallsetminus\{\mathds{1}_N\}$ is the orbit of the abelian character $\varphi_2$ under the action of $H$, hence $N$ is abelian.


Finally since $H$ acts transitively by conjugation on $N\smallsetminus\{1\}$ and since $N$ is normal in $G$ we deduce that
$N\smallsetminus\{1\}$ is a conjugacy class of $G$. 
\end{proof}


In~\cite{FJ}*{\S 9.2} and in~\cite{BFJ}*{\S 2.2}, a particular family of Galois extensions of $\Q$ having a doubly transitive Frobenius group as Galois group is considered. We recall its description in the example below.


\begin{ex} \label{ex:rad}
Let $a,p$ be distinct prime numbers such that $p\neq 2$ and $a^{p-1}\not\equiv 1\bmod p^2$ and let $K_{a,p}$ be the splitting field inside $\C$ of $X^p-a\in\Q[X]$. Then $G:=\Gal(K_{a,p}/\Q)$ is isomorphic to the group of transformations of the affine line $\mathbb{A}^\mathds{1}_{\F_p}$; it is a doubly transitive Frobenius group. 
Indeed $G$ can be described as follows:
$$
G\simeq  \left\{\left(\begin{array}{cc}
c & d \\
0 & 1
\end{array} \right)\colon c\in \F_p^*, d\in\F_p\right\}\,.
$$
Here the Frobenius kernel $N$ and the Frobenius complement $H$ are respectively
$$
N:=\left\{\left(\begin{array}{cc}
1 & \star \\
0 & 1
\end{array} \right)\right\}\,,
\qquad 
H=\left\{\left(\begin{array}{cc}
\star & 0 \\
0 & 1
\end{array} \right)\colon \star\in\F_p^\times\right\}\,.
$$
The group $G$ admits $p$ irreducible characters: $p-1$ of them are lifts of the characters of $H$ \textit{via} the isomorphism $G/N\simeq H$. The unique non abelian irreducible character of $G$ has degree $p-1$ and is non zero only on two conjugacy classes: $\{1\}$ and $N\smallsetminus\{1\}$.



\end{ex}

 The last property mentioned in the example is more generally enjoyed by any doubly transitive Frobenius group $G$ (as seen by combining Proposition~\ref{prop:snyder} and Proposition~\ref{prop:carGpFrob}(2)): the unique non abelian irreducible character of $G$ vanishes on all but two conjugacy classes. This property, which has been studied by group theorists for independent reasons (see \textit{e.g.}~\cite{Gag}), implies in our case a general form of orthogonality of characters, as we now show.

\begin{proposition}\label{prop ortho pour G=d(d+1)}
Let $G$ be a doubly transitive Frobenius group of order $d(d+1)\geq 6$ with kernel $N$ and abelian complement $H$ of order $d$. Denote by $\{1=\psi_1,\ldots,\psi_d\}$ the irreducible characters of $H$, which all have degree $1$ (by abuse of notation we will see these characters as characters of degree $1$ of $G$). Finally denote by $\vartheta$ the unique non abelian character of $G$. It satisfies
$$
\vartheta(1)=d\,, \qquad \vartheta(N\smallsetminus\{1\})=-1\,,\qquad \big(g\in G\smallsetminus N\Rightarrow \vartheta(g)=0\big)\,.
$$ 
In particular, for any $n\geq 1$ and any choice of irreducible characters $\chi_1,\ldots,\chi_n$ (with possible repetitions) of $G$, one has the generalized orthogonality relations
\begin{equation}\label{eq:orthoFrobenius}
\sum_{C\in G^\sharp}|C| \chi_1(C) \cdots \chi_n(C) = \left.
  \begin{cases}
    |G|\delta(\chi_1\cdots\chi_n=1), & \text{if } \chi_i\neq \vartheta\,,\text{ for all }i\,, \\
    d^k+(-1)^kd, & \text{if } \chi_i= \vartheta\,,\text{ for $k\geq 1$ values of }i\,.
  \end{cases}\right.
\end{equation}
\end{proposition}

\begin{proof}
This is a consequence of Corollary~\ref{cor:doubtrans} together with the fact that $\langle\vartheta,\vartheta\rangle_G=1$ and the equality $\psi_j(N)=\{1\}$ which holds for all $j$.
\end{proof}

For a Galois extension of number fields $L/K$ such that $G=\Gal(L/K)$ is a doubly transitive Frobenius group with a unique non abelian character $\vartheta$, we will now bound $(-1)^n M_{n,1}(U;L/K,\Phi,\eta)$ in terms of the sums
$$ S_{n,k}(L/K,\eta):= \!\!\!\!\!\!\!\!\!\sum_{\substack{  \chi_1,\dots,\chi_{n-k} \in \Irr(G) \\ \forall i, \chi_i \neq \vartheta  }}\sum_{\substack{ \gamma_{\chi_1},\dots,\gamma_{\chi_{n-k}} , \gamma_{\vartheta,1},\dots,\gamma_{\vartheta,k}  \neq 0\\ \gamma_{\chi_1}+\dots+\gamma_{\chi_{n-k}} + \gamma_{\vartheta,1}+\dots+\gamma_{\vartheta,k} = 0}}   \prod_{i=1}^{n-k} \widehat \eta\Big( \frac {\gamma_{\chi_i}}{2\pi} \Big) \prod_{j=1}^{k} \widehat \eta\Big( \frac {\gamma_{\vartheta,j}}{2\pi} \Big) \qquad (k\leq n-1);$$
$$ S_{n,n}(L/K,\eta):=\sum_{\substack{  \gamma_{\vartheta,1},\dots,\gamma_{\vartheta,n}  \neq 0\\ \gamma_{\vartheta,1}+\dots+\gamma_{\vartheta,n} = 0}}    \prod_{j=1}^{n} \widehat \eta\Big( \frac {\gamma_{\vartheta,j}}{2\pi} \Big) .$$

\begin{lemma} \label{lem:minM1}
Let $G$ be a doubly transitive Frobenius group of order $d(d+1) { \geq 6}$ with kernel $N$ and abelian complement $H$ of order $d$. Let $L/K$ be a Galois extension of number fields of group $G$, and assume GRH. Let $\eta\in \mathcal S_\delta$, $\Phi\in \mathcal U$. For $U\geq 1$, we have the following bound: 
\begin{multline}
 (-1)^n M_{n,1} (U;L/K,\Phi,\eta)\geq    \sum_{\substack{ \chi_1,\dots,\chi_{n} \in \Irr(G) \\ \forall i, \chi_i \neq \vartheta \\ \chi_1\cdots \chi_n =1  }}\sum_{\substack{ \gamma_{\chi_1},\dots,\gamma_{\chi_{n}} \neq 0\\ \sum_{i=1}^{n} \gamma_{\chi_i} = 0}}    \prod_{i=1}^{n} \widehat \eta\Big( \frac {\gamma_{\chi_i}}{2\pi} \Big) \\
 +\sum_{k=1}^n \binom nk\frac{ d^{k-1}+(-1)^k}{d+1}  S_{n,k}(L/K,\eta) +O_{\Phi}\Big(\frac{(d+1)^{2n-2} (\kappa_\eta[K:\Q] \log(\rd_L+2))^n}U\Big),
 \label{eq:lemma 3.5}
 \end{multline}
where $\vartheta$ denotes the unique non-abelian character of $G$, and the $\gamma_{\vartheta,j}$ run through the imaginary parts of the non-trivial zeros of $L(s,L/K,\vartheta)$. Here, $\kappa_\eta>0$ is a constant which depends on $\eta$. In particular, since all summands in this expression are non-negative, it follows that for $n=2m$,
\begin{equation}
   M_{2m,1} (U;L/K,\Phi,\eta)\geq
 \frac{ d^{2m-1}+1}{d+1}  S_{2m,2m}(L/K,\eta) +O_{\Phi}\Big(\frac {(d+1)^{4m-2} (\kappa_\eta[K:\Q] \log(\rd_L+2))^{2m}}U\Big).
 \label{eq:lemma 3.5 simplifiee}
 \end{equation}
\end{lemma}

\begin{proof}

Applying orthogonality of characters in the form
$$\psi_\eta(\e^u;L/K,\mathds{1}_C)
=\frac{|C|}{|G|}
\sum_{\chi \in \Irr(G)} \overline{\chi}(C)\psi_\eta(\e^u;L/K,\chi),$$
we obtain the expression 
\begin{align*} \M_n(u;L/K,\eta)   &= \frac 1{|G|} \sum_{C\in G^\sharp} |C|\big(\psi_\eta^*(\e^u,L/K,t_C)
\big)^n
\cr&= \frac 1{|G|}
\sum_{ \chi_1,\dots,\chi_n \in \Irr(G)} \prod_{j=1}^n   \psi^*_\eta(\e^u,L/K,\chi_j) \sum_{C\subset G^\sharp} |C|\overline{\chi_1}(C)\cdots \overline{\chi_n}(C),
\end{align*}
where we recall that $t_C=\frac{|G|}{|C|}\mathds{1}_C$ and the definition~\eqref{eq:psi*} of $\psi^*_\eta$.

 Combining this formula with Proposition~\ref{prop ortho pour G=d(d+1)}, it follows that 
\begin{align*} \M_n&(u ;L/K,\eta)
=\sum_{\substack{ \chi_1,\dots,\chi_n \in \Irr(G) \\ \forall i, \chi_i \neq \vartheta   \\ \chi_1\cdots\chi_n = 1 }} \prod_{j=1}^n \psi^*_\eta(\e^u,L/K,\chi_j)  \cr&  
+ \frac{1}{|G|} \sum_{k=1}^n \binom nk \psi^*_\eta(\e^u,L/K,\vartheta)^k  \sum_{\substack{\chi_1,\dots,\chi_{n-k} \in \Irr(G) \\ \forall i, \chi_i \neq \vartheta  }} \prod_{j=1}^{n-k} \psi^*_\eta(\e^u,L/K,\chi_j)    \big(d^k+(-1)^kd\big)\,, 
\end{align*} 
where by convention the innermost sum is equal to $1$ when $k=n$.
Hence,
\begin{align}
     (-1)^nM_{n,1}   (U;L/K,\Phi,\eta) 
   = M_{n,1}^{\mathrm{main}}(U;L/K,\Phi,\eta)
 +M_{n,1}^{\mathrm{err}}(U;L/K,\Phi,\eta)\,, \label{eq:seperation M1n two sums}\end{align}
 where 
 \begin{multline*}
      M_{n,1}^{\mathrm{main}}(U;L/K,\Phi,\eta):= \frac {(-1)^n}{|G|  \int_0^\infty \Phi} \sum_{k=1}^n \binom nk \big(d^k+(-1)^kd\big) \times \\ \int_0^\infty \Phi(\tfrac uU)\psi^*_\eta(\e^u,L/K,\vartheta)^k \sum_{\substack{ \chi_1,\dots,\chi_{n-k} \in \Irr(G) \\ \forall i, \chi_i \neq \vartheta  }}\prod_{j=1}^{n-k} \psi^*_\eta(\e^u,L/K,\chi_j)      \frac{\d u}{U}\,,
 \end{multline*}
 \[
 M_{n,1}^{\mathrm{err}}(U;L/K,\Phi,\eta):=\frac {(-1)^n}{  \int_0^\infty \Phi}\int_0^\infty \Phi(\tfrac uU) \sum_{\substack{ \chi_1,\dots,\chi_n \in \Irr(G) \\ \forall i, \chi_i \neq \vartheta   \\ \chi_1\cdots\chi_n =  1 }}\prod_{j=1}^{n} \psi^*_\eta(\e^u,L/K,\chi_j)  \frac{\d u}{U}\,.
 \]

 We recall that~\cite{BFJ}*{Lemma 4.1} implies the estimate
$$
\psi^*_\eta(\e^u;L/K,\chi) =  - \sum_{\gamma_\chi\neq 0} \e^{i\gamma_\chi u} \widehat \eta\Big( \frac {\gamma_\chi}{2\pi} \Big)+O_\eta\big(\e^{-\frac u2} \log (A(\chi)+2)\big),
$$
as well as the bound (see~\cite{BFJ}*{Lemmas 3.1 and 4.2})
$$
\psi^*_\eta(\e^u;L/K,\chi) \ll_\eta  \log (A(\chi)+2)\ll\chi(1) [K:\Q] \log(\rd_L+2).
$$

Applying these estimates, it follows that 
\begin{multline}
 M_{n,1}^{\mathrm{main}}(U;L/K,\Phi,\eta)= \frac {1}{|G| \int_0^\infty \Phi} \sum_{k=1}^n \binom nk(d^k+(-1)^kd) \int_0^\infty \Phi(\tfrac uU)\\ \sum_{\substack{ \chi_1,\dots,\chi_{n-k} \in \Irr(G) \\ \forall i, \chi_i \neq \vartheta  }}\sum_{\substack{ \gamma_{\chi_1},\dots,\gamma_{\chi_{n-k}} \neq 0\\ \gamma_{\vartheta,1},\dots,\gamma_{\vartheta,k}  \neq 0}}  \!\!\!\!\! \e^{i(\sum_{j=1}^k\gamma_{\vartheta,j}+\sum_{i=1}^{n-k} \gamma_{\chi_i}) u}  \prod_{i=1}^{n-k} \widehat \eta\Big( \frac {\gamma_{\chi_i}}{2\pi} \Big) \prod_{j=1}^{k} \widehat \eta\Big( \frac {\gamma_{\vartheta,j}}{2\pi} \Big)   \frac{\d u}{U}+E, 
  \label{eq:second term}
\end{multline}
where the error term $E$ satisfies the bound 
\begin{align*}
   |M_{n,1}^{\mathrm{err}}(U;L/K,\Phi,\eta)|+ |E|&\ll_{\Phi} U^{-1}(\kappa_\eta[K:\Q] \log(\rd_L+2))^n\sum_{\substack{ \chi_1,\dots,\chi_n \in \Irr(G) \\ \forall i, \chi_i \neq \vartheta   \\ \chi_1\cdots\chi_n = 1 }}1 
    \\ &\quad+ (|G|U)^{-1} \sum_{k=1}^n\binom nk d^{2k} (\kappa_\eta[K:\Q] \log(\rd_L+2) )^{n}  \sum_{\substack{ \chi_1,\dots,\chi_{n-k} \in \Irr(G) \\ \forall i, \chi_i \neq \vartheta  }}1
    \\ &\ll  U^{-1} (d+1)^{2n-2} \big(\kappa_\eta[K:\Q] \log(\rd_L+2)\big)^n.
\end{align*} 
{We observe that, as $\Phi\in \mathcal U$,  $ \Phi$ is even and
$$ \int_0^\infty \Phi\Big(\frac uU\Big) \e^{i\gamma u} \frac {\d u}U=\frac 12 \widehat \Phi\Big(\frac{U\gamma}{2\pi}\Big). $$}
Then, the main term in~\eqref{eq:second term} is 
\begin{multline*}
\geq   \frac {1}{ 2 \int_0^\infty \Phi} \sum_{k=1}^n \binom nk\frac{ d^{k-1}+(-1)^k}{d+1}  \\ \sum_{\substack{ \chi_1,\dots,\chi_{n-k} \in \Irr(G) \\ \forall i, \chi_i \neq \vartheta  }}\sum_{\substack{ \gamma_{\chi_1},\dots,\gamma_{\chi_{n-k}} \neq 0\\ \gamma_{\vartheta,1},\dots,\gamma_{\vartheta,k}  \neq 0}}  \!\!\!\!\! \widehat \Phi\Big(\frac{U}{2\pi}\Big( \sum_{j=1}^k\gamma_{\vartheta,j}+\sum_{i=1}^{n-k} \gamma_{\chi_i}\Big)  \Big)  \prod_{i=1}^{n-k} \widehat \eta\Big( \frac {\gamma_{\chi_i}}{2\pi} \Big) \prod_{j=1}^{k} \widehat \eta\Big( \frac {\gamma_{\vartheta,j}}{2\pi} \Big), 
\end{multline*}
which by positivity of $\widehat \Phi$ is greater than or equal to the second term in~\eqref{eq:lemma 3.5}. Arguing similarly for the first term on the right hand side of~\eqref{eq:seperation M1n two sums} yields the claimed estimate.


\end{proof}

In the next step we will bound the quantity $S_{2m,2m}(L/K,\eta)$ which appears in~\eqref{eq:lemma 3.5 simplifiee} through a second application of positivity. The key idea which will guide us in this task is that we believe that in the case where $L/\Q$ is Galois, the multiset of imaginary parts of non-trivial zeros of $\prod_{\chi \in \Irr(\Gal(L/\Q))} L(s,L/\Q,\chi)$ is linearly independent over the rationals; this is known as hypothesis LI. With this in mind, we will drop all the terms for which LI does not hold; in particular, we will only consider the case where $n$ and $k$ are even.
In the case where $K \neq \Q$, we refer the reader to~\cite{FJ}*{Section 1.2} for possible counterexamples to LI.

\begin{lemma}\label{lem:minS}
Let $G$ be a doubly transitive Frobenius group of order $d(d+1) \geq 6$ with kernel $N$ and abelian complement $H$ of order $d$. Let $\vartheta$ be the unique non-abelian irreducible character of $G$, let $L/K$ be a Galois extension of number fields of group $G$, and assume GRH. Let $\eta\in \mathcal S_\delta$. If $  m\in \mathbb N$, then we have the bound
$$
 S_{2m,2m}(L/K,\eta)\geq
\mu_{2m}
   b_0(\vartheta, \widehat \eta^2)^{m-1} 
\max\Big\{ b_0 (\vartheta, \widehat \eta^2)- 2m!m(m-1)||\widehat
   \eta||_{\infty}^2,0\Big\}. 
 $$

\end{lemma}

\begin{proof}
Using positivity, thanks to \cite{BFJ}*{Lemma 5.3},  we have
  \begin{align*}
  S_{2m,2m}(L/K,\eta)&=
\sum_{\substack{  \gamma_{\vartheta,1},\dots,\gamma_{\vartheta,2m}  \neq 0\\  \gamma_{\vartheta,1}+\dots+\gamma_{\vartheta,2m} = 0}}     
\prod_{j=1}^{2m} \widehat \eta\Big( \frac {\gamma_{\vartheta,j}}{2\pi} \Big)\cr
&\geq \binom {2m}{m}
\sum_{\substack{  \gamma_{\vartheta,1},\dots,\gamma_{\vartheta,2m}  \neq 0\\  \gamma_{\vartheta,j}>0 \,(\forall j\leq m) ,\, \gamma_{\vartheta,j}<0 \,(\forall j> m) \\ (\forall \gamma\in \mathbb R)\,\,
|\{ j : \gamma_{\vartheta,j}=\gamma\}|=|\{ j : \gamma_{\vartheta,j}=-\gamma\}|}}     \prod_{j=1}^{2m} \widehat \eta\Big( \frac {\gamma_{\vartheta,j}}{2\pi} \Big)\cr &
\geq \mu_{2m} b_0 (\vartheta, \widehat \eta^2)^{m-1}\max\Big\{ b_0 (\vartheta, \widehat \eta^2)- 2m!m(m-1)||\widehat
   \eta||_{\infty}^2,0\Big\}.
  \end{align*} 
  \end{proof}

To conclude the proof of Theorem~\ref{theorem main 2}, we will need the following bound.

 \begin{lemma}\label{lem:condGpFrob}
 Let $G$ be a doubly transitive Frobenius group of order $d(d+1) \geq 6$ with kernel $N$ and abelian complement $H$ of order $d$. Let $L/K$ be a Galois extension of number fields of group $G$.
Then one has the following lower bound:
$$ \log A(\vartheta) \geq (d-1) [K:\Q] \log (\rd_L). $$ 
%
\end{lemma}
 
 \begin{proof}
 Define for any non zero class function $t\colon G\to \C$ (see~\cite{BFJ}*{(10)}),
 $$
 S_t=\frac{\max_{a\neq 1}\big|\sum_{\chi\in\Irr(G)}|\widehat{t}(\chi)|^2\chi(a)\big|}{\lambda_{1,2}(t)}\,,
 $$
 { where $\lambda_{1,2}(t)=\sum_{\chi\in\Irr(G)}\chi(1)|\widehat{t}(\chi)|^2$.}
 One has $S_\vartheta=\frac 1d$ by the definition and properties of $\vartheta$ (see Proposition~\ref{prop ortho pour G=d(d+1)}). The bound is then obtained by applying~\cite{BFJ}*{Lemma 3.2}.
 \end{proof}

We are now ready to prove Theorem~\ref{theorem main 2}.
  
  \begin{proof}[Proof of Theorem~\ref{theorem main 2}]
  
  Combining Lemmas~\ref{lem:minM1} and~\ref{lem:minS}, we deduce that
  \begin{multline*}
   M_{2m,1} (U;L/K,\Phi,\eta)\geq
 \frac{ d^{2m-1}+1}{d+1} \mu_{2m}
   b_0(\vartheta, \widehat \eta^2)^{m-1} 
\max\Big\{ b_0 (\vartheta, \widehat \eta^2)- 2m!m(m-1)||\widehat
   \eta||_{\infty}^2,0\Big\}
   \\ +O_{\Phi}\Big( \frac{(d+1)^{4m-2} (\kappa_\eta[K:\Q] \log(\rd_L+2))^{2m}}U \Big).
 \end{multline*}

Applying~\cite{BFJ}*{Lemma 4.2} and~\cite{FJ}*{Lemma 4.1} to estimate $b_0 (\vartheta, \widehat \eta^2)$, we see that the first term in this expression is
  \begin{align*} 
 &\geq \frac{\mu_{2m} d^{2m}}{d(d+1)} \widehat \eta^{2m}(0)  \Big( 
  \log A(\vartheta )  +O_{\eta}(d[K:\Q])\Big)^{m }
  \\
  & 
  \geq  \frac{\mu_{2m} d^{2m}}{d(d+1)} \widehat \eta^{2m}(0) \Big((d-1) [K:\Q]\log (\rd_L)\Big)^m\Big(1+O_{\eta}\Big(\frac m{\log(\rd_L+2) } \Big)\Big),
  \end{align*}
  by Lemma~\ref{lem:condGpFrob}. The first claimed result follows.
  
  As for the second, note that in the limit as $U\rightarrow \infty$, the variance $M_{2,1}(U;L/K,\Phi,\eta)$ is asymptotically at most $d^2 [K:\Q]^2\log(\rd_L)^2(1+o_{d\rightarrow \infty}(1))$ (combine~\eqref{eq:parseval} with~\cite{BFJ}*{Lemma 4.1}). The claimed bound follows from combining this upper bound with the lower bound~\eqref{eq:borneinfTh1.2}.

   \end{proof}

\section*{Acknowledgements}
The work of the third author was partly funded by the ANR through project FLAIR (ANR-17-CE40-0012). The authors would like to thank {\it Villa La Stella} in Florence for their hospitality and excellent working conditions during a stay in March 2022 where substantial parts of this work were accomplished.

\end{document}